\documentclass[oneside,reqno,11pt]{amsart}

\usepackage{amsthm}
\usepackage{amssymb}
\usepackage{bbm}
\usepackage{url}
\usepackage{graphicx,xcolor}
\usepackage{amsaddr}

 \newtheorem{thrm}{Theorem}[]
 \newtheorem*{thrm*}{Theorem}
 \newtheorem{lemma}[thrm]{Lemma}
 \newtheorem*{lemma*}{Lemma}
 \newtheorem*{Vermutung*}{Vermutung}
 
 \newtheorem{prop}[thrm]{Proposition}
 \theoremstyle{definition}
 
\newtheorem*{defi*}{Definition}

  \theoremstyle{remark}

  \newtheorem{remark}[thrm]{Remark}
  \newtheorem*{remark*} {Remark}

\newcommand{\dopp} [1] {\mathbbmss{#1}}
\newcommand{\m}{\mathbb}
\newcommand{\f}{\begin{align}}
\newcommand{\fo}{\begin{align*}}
\newcommand{\fe}{\end{align}}
\newcommand{\foe}{\end{align*}}
\newcommand{\map}[3] {#1:\, #2\, \longrightarrow \, #3}
\renewcommand{\P}{\m{P}}

\newcommand{\R}{\mathbb{R}}

\newcommand{\E}{\mathbb{E}}

\newcommand{\U}{\mathcal{U}}

\theoremstyle{plain}

\theoremstyle{definition}

\theoremstyle{remark}

\renewcommand{\a}{\alpha}
\renewcommand{\b}{\beta}

\newcommand{\e}{\eta}

\renewcommand{\l}{\lambda}

\renewcommand{\o}{\^{o}}

\renewcommand{\hat}{\widehat}

\renewcommand{\phi}{\varphi}
\renewcommand{\e}{\epsilon}
\renewcommand{\epsilon}{\varepsilon}

\newcommand{\Ph}{P_{N,Q,\b}^h}

\newcommand{\F}{\widehat}
\renewcommand{\U}{\mathcal{U}}
\newcommand{\cir}[1]{\overset{\circ}{#1}}

\newcommand{\lv}{\left\lvert}
\newcommand{\rv}{\right\rvert}

\newcommand{\Lip}[1]{\left\lvert #1 \right\rvert_{\mathcal{L}}}

\newcommand{\lb}{\left(}
\newcommand{\rb}{\right)}

\renewcommand{\o}{\omega}

\newcommand{\w}{\omega}

{\setlength\arraycolsep{2pt}
\setlength{\parindent}{0.5cm}

\newcommand{\lvb}{\big\lvert}
\newcommand{\rvb}{\big\rvert}
\newcommand{\lbb}{\big(}
\newcommand{\rbb}{\big)}
\newcommand{\lee}{\big\{}
\newcommand{\ree}{\big\}}

\begin{document}

\title[]{Particle Systems with Repulsion Exponent $\b$ \\and Random Matrices}
\author[Martin Venker]{\small{Martin Venker}}
\address{\small{Department of Mathematics, Bielefeld University,\\P.O.Box 100131, 33501 Bielefeld, Germany}}
\email{mvenker@math.uni-bielefeld.de}
\thanks{Research supported by the CRC 701 ``Spectral Structures and Topological Methods in Mathematics''.}


\begin{abstract}
We consider a class of particle systems generalizing the $\b$-Ensembles from random matrix theory. In these new ensembles, particles
experience repulsion of power $\b>0$ when getting close, which is the same as in the $\b$-Ensembles. For
distances larger than zero, the interaction is allowed to differ from those present for random
eigenvalues. We show that the local bulk correlations of the $\b$-Ensembles, universal in random matrix theory, also appear in
these new ensembles. 
\end{abstract}

 \maketitle

\section{Introduction and Main Results}

A central theme in random matrix theory is the universality phenomenon, which means that many essentially
different matrix distributions lead in the limit of growing dimension to the same spectral statistics.

In the past 15 years or so, much progress has been made in proving universality of local spectral distributions, especially
correlations between neighboring eigenvalues in the bulk of the spectrum and of the largest eigenvalues. It is known that
there is a parameter, usually denoted $\b$, which determines the universality class of the ensemble. To explain this in more detail,
define for any $\b>0$ and  a
continuous
function $\map{Q}{\R}{\R}$ the invariant $\b$-Ensemble
$P_{N,Q,\b}$ on $\R^N$ which is given by
\begin{align}
 P_{N,Q,\b}(x):=\frac{1}{Z_{N,Q,\b}}\prod_{i<j}\lv x_i-x_j\rv^\b e^{- N\sum_{j=1}^NQ(x_j)}.\label{beta3}
\end{align}
(With a slight abuse of notation, we will not distinguish between a measure and its density.) Here we assume $Q(t)\geq \b'\log\lv
t\rv$ for $\lv t\rv$ large enough for some $\b'\geq \b$
with $\b'>1$.

For $\b=1,2,4$, $P_{N,Q,\b}$ is the eigenvalue distribution of a probability ensemble on the space of
real symmetric ($\b=1$), complex Hermitian ($\b=2$) or quaternionic self-dual ($\b=4$) $(N\times N)$ matrices,
respectively. The matrix distributions are invariant under orthogonal, unitary or symplectic conjugations,
respectively, explaining the name ``invariant ensembles''. For arbitrary $\b$, only for quadratic $Q$, $P_{N,Q,\b}$ is known to be
an eigenvalue distribution.

It has been shown (see \cite{GoetzeVenker} for references) that the local spectral statistics in the bulk or at the edges of the
spectrum do in many cases not
depend on $Q$ or, in other terms, invariant ensembles with different potentials $Q$ but the same $\b$ have the same
local statistics. It is also known that different values of $\b$ lead to different limiting (local) distributions. This is
not surprising as the interaction term $\prod_{i<j}\lv x_i-x_j\rv^\b$ has a strong effect on neighboring eigenvalues whereas
$e^{- N\sum_{j=1}^NQ(x_j)}$ just confines all eigenvalues independently into a compact interval. In the limit $N\to\infty$, these
two competing forces balance and produce a limiting measure of compact support.

In \cite{GoetzeVenker} the question was addressed whether the interaction term $\prod_{i<j}\lv x_i-x_j\rv^\b$ could be changed
without changing the local statistics. To this end, we introduced ensembles with density proportional to

\begin{align}
\prod_{i<j}\phi(x_i-x_j)e^{- N\sum_{j=1}^NQ(x_j)},\label{beta4}
\end{align} 
where $Q$  is a continuous function of sufficient growth at infinity compared to the continuous function
$\map{\phi}{\R}{[0,\infty)}$. The interaction potential $\phi$ fulfills 
\begin{align*}
\phi(0)=0,\quad \phi(t)>0\ \text{for }t\not=0\quad \text{and }\ \lim_{t\to0}\frac{\phi(t)}{\lv t\rv^\b}=c>0\quad \text{ for some
}\b>0,
\end{align*}
or, in other terms, $0$ is the only zero of $\phi$ and it is of order $\b$. It has been conjectured in \cite{GoetzeVenker} that the
bulk
correlations for the ensembles \eqref{beta4} are the same as in the case $\phi(t)=\lv t\rv^\b$, i.e. the same as for the invariant
ensembles in random matrix theory. This was proved in \cite{GoetzeVenker} for $\b=2$ and a special class of functions $\phi$
and $Q$. In the present work, we prove a similar result for arbitrary $\b>0$. This shows that the local bulk correlations (at
least in the considered cases) merely depend on the \textit{repulsion exponent} $\b$ and not on the interaction of particles at
distances larger than $0$.

We believe that these results may lead to an explanation for the occurrence of random matrix bulk statistics in a number of seemingly
unrelated observations in real world and science (see \cite{GoetzeVenker} for references). Spacings between cars in different
situations were found to be fitted well by the universal spacing
statistics from random matrix theory ($\b=1$ for parking along one-way streets, $\b=2$ along two-way streets, $\b=4$ for waiting
in front of traffic signals). Also spacings between perching birds and between bus arrival times at stops in certain cities seem to
obey ($\b=2$) random matrix spacing statistics. Gaps between zeros of the Riemann zeta function on the critical line are another
famous
example from mathematics (also $\b=2$).
In
all these
observations,
a strong repulsion between consecutive quantities is present. 

Furthermore, the ensemble \eqref{beta4} does not seem to have a natural spectral
interpretation which makes our findings a first step in proving universality of random matrix bulk distributions for more
general
particle
systems.

To state our main results, we first rewrite the ensemble \eqref{beta4}. Let $h$ be a continuous even function which is bounded
below. Let $Q$ be
a continuous even function of sufficient growth at infinity. By $\Ph$ we will denote the probability density on $\R^N$ defined
by
 \begin{align} 
  \Ph(x):=\frac{1}{Z_{N,Q,\b}^h}\prod_{i<j}\lv x_i-x_j\rv^\b\exp\{- N\sum_{j=1}^NQ(x_j)-\sum_{i<j}h(x_i-x_j)\} , \label{211}
 \end{align}
where $Z_{N,Q,\b}^h$ denotes the normalizing constant. The density $\Ph$ can also be written in the form \eqref{beta4} with
$\phi(t):=\lv t\rv^\b\exp\{-h(t)\}$.

Furthermore, let for a probability density $P_N$ on $\R^N$ and $k=1,2,\dots$,
\begin{align*}
 \rho_{N}(t_1,\dots,t_k):=\int_{\R^{N-k}}P_N(t_1,\dots,t_k,x_{k+1},\dots,x_N)\,dx_{k+1}\dots dx_N
\end{align*}
denote the $k-$th correlation function of $P_N$. The correlation functions are the marginal densities. The measure
$\rho_{N}(t_1,\dots,t_k)\,dt_1\dots dt_k$ is called $k-$th correlation measure. Denote by $\rho_{N,Q,\b}^{h,k}$ the $k$-th
correlation function of $\Ph$ and by $\rho_{N,Q,\b}^k$ the $k$-th correlation function of $P_{N,Q,\b}$ from \eqref{beta3}.
Universality of ensembles is usually defined by universality of their correlation functions or measures as many interesting
statistics of the ensembles can be expressed in terms of correlation functions.
Finally, introduce for a twice differentiable convex function $Q$ the quantity $\a_Q:=\inf_{t\in\R}Q''(t)$.

The following theorem deals with the global or macroscopic behavior of the ensemble $\Ph$.

\begin{thrm}\label{global}
Let $h$ be a real analytic and even Schwartz function. Then there
exists a constant $\a^h\geq 0$ such that for all real analytic, strongly convex and even $Q$ with $\a_Q> \a^h$, the following
holds:\\
The first correlation measure $\rho_{N,Q,\b}^{h,1}$ converges weakly to a compactly supported probability measure
$\mu_{Q,\b}^h$ which has a non-zero and continuous density on the interior of
its
support. Weak convergence means that for any bounded and
continuous $\map{f}{\R}{\R}$, we have
\begin{align*}
	\lim_{N\to \infty}\int f(t)\, \rho_{N,Q,\b}^{h,1}(t)\, dt=\int f(t)\, \mu_{Q,\b}^h(t)dt.
\end{align*}
\end{thrm}

\begin{remark*}\noindent
\begin{itemize}
\item In general, $\mu_{Q,\b}^h$ depends on $h$, i.e. changing the interaction term has an influence on the (limiting) global
density of the particles.
\item If $h$ is positive semi-definite, then $\a^h$ in Theorem \ref{global} may be explicitly chosen as
$\a^h=\sup_{t\in\R}-h''(t)$.
\item For $k=2, 3, \dots,$ the $k-$th correlation measure converges weakly to the $k$-fold product
$\lbb\mu_{Q,\b}^h\rbb^{\otimes
k}$. This has been shown in \cite{GoetzeVenker} for $\b=2$ but the same proof goes through for arbitrary $\b>0$. However, as a
byproduct of the local universality result, the proof in \cite{GoetzeVenker} uses some rather technical and complicated arguments
which we have no further use for in this article. We will therefore give a short proof of Theorem \ref{global} which only uses
methods needed anyway.
\item Note that the dependence of $\mu_{Q,\b}^h$ on $\b$ can be eliminated if the prefactor $\b$ is put in
front of $Q$ and $h$.
\item In \cite{BPS}, ensembles with many-body interactions are considered, replacing $h$ in \eqref{211}. Here global asymptotics
 but not local correlations are discussed. In the case of pair interactions, the classes of admissible interactions in \cite{BPS}
and
in this paper are different. In \cite{BPS}, a convexity condition is posed, depending solely on the additional interaction
potential where our conditions depend on both $Q$ and $h$. The characterisation of the limiting measure is different, too.

In \cite{Chafaietal}, a large deviations principle has been shown for interacting particle systems of the type \eqref{beta4} in
$\R^d, d\geq1$ and without specification of the repulsion behavior.
\end{itemize}
\end{remark*}

The next theorem states the local universality in the bulk. We use the notion of universality by Bourgade, Erd\H{o}s, Schlein, Yau,
Yin et al. (see e.g.
\cite{ErdosYau12} and the references therein). Let $G$ be the Gaussian potential $G(t):=x^2$ and recall that the corresponding
limiting measure $\mu_{Q,\b}$
is the semicircle distribution (with a certain variance depending on $\b$). Recall that under mild assumptions on $Q$, there is a
measure $\mu_{Q,\b}$ of compact support which is the weak limit of the first correlation measure of $P_{N,Q,\b}$. Consider the
scaled correlation functions 
\begin{align}
\frac{1}{\mu_{Q,\b}^h(a)^k}\rho_{N,Q,\b}^{h,k}\lb
a+\frac{t_1}{N\mu_{Q,\b}^h(a)},\dots,a+\frac{t_k}{N\mu_{Q,\b}^h(a)}\rb,\label{beta6}
\end{align}
where $a$ is a point with $\mu_{Q,\b}^h(a)>0$ and $t_1,\dots,t_k$ are contained in an $N$-independent compact interval. Under this
scaling, the local density around $a$ will be asymptotically one, in particular independent of $a$. For $N\to\infty$, $h=0$ and
$Q=G$,
the limit of
\eqref{beta6} exists and has been described in terms of a stochastic process in \cite{ValkoVirag}. As for general $\b$ no
nice formula for this limit is known, we state the following theorem as universality result, comparing the local correlations of
$\Ph$ with
those of the Gaussian $\b$-Ensemble $P_{N,G,\b}$.

\begin{thrm}\label{local}
Let $h$ and $Q$ satisfy the assumptions of Theorem \ref{global}. Let $0<\xi\leq 1/2$ and set $s_N:=N^{-1+\xi}$. Then for
$k=1,2,\dots$, we have for any $a$ in the interior of the support of $\mu_{Q,\b}^h$, any $a'$ in the interior of the support of the
semicircle law $\mu_{G,\b}$ and any smooth function $\map{f}{\R^k}{\R}$ with compact support
\begin{align*}
\lim_{N\to\infty} \int &f(t_1,\dots,t_k)\Bigg[\int_{a-s_N}^{a+s_N}\frac{1}{\mu_{Q,\b}^h(a)^k}\rho_{N,Q,\b}^{h,k}\lb
u+\frac{t_1}{N\mu_{Q,\b}^h(a)},\dots,u+\frac{t_k}{N\mu_{Q,\b}^h(a)}\rb\,\frac{du}{2s_N}\\
-&\int_{a'-s_N}^{a'+s_N}\frac{1}{\mu_{G,\b}(a')^k}\rho_{N,G,\b}^{k}\lb
u+\frac{t_1}{N\mu_{G,\b}(a')},\dots,u+\frac{t_k}{N\mu_{G,\b}(a')}\rb\,\frac{du}{2s_N}\Bigg]dt_1\dots dt_k\\
&=0.
\end{align*}
\end{thrm}

\begin{remark*}\noindent
\begin{itemize}
\item If the inner integrations were not present, the convergence in Theorem \ref{local} would be vague convergence of the scaled
correlation measures. Here an
additional small (uniform) average around the points $a$ and $a'$ is performed.
\item If $h$ is positive semi-definite, then $\a^h$ in Theorem \ref{local} may be explicitly chosen as
$\a^h=\sup_{t\in\R}-h''(t)$.
\item The choice of the Gaussian $\b$-Ensemble $P_{N,G,\b}$ is just for definiteness, in fact any other ensemble belonging to the
same universality class could be chosen. So far, these are known to be basically all $P_{N,Q,\b}$ with the same $\b$ and real
analytic $Q$ which leads to a limiting measure $\mu_{Q,\b}$ of connected support \cite{Bourgadeetal}.
\end{itemize}
\end{remark*}

These results should be compared to those of \cite{GoetzeVenker}. There we could show for $\b=2$ under the same conditions
on $Q$ and $h$ a much stronger type of convergence as in Theorem \ref{local}. We proved in \cite{GoetzeVenker}

\begin{align*}
 &\lim_{N\to\infty}\frac{1}{\mu_{Q,\b=2}^h(a)^k}\rho_{N,Q,\b=2}^{h,k}\lb
a+\frac{t_1}{N\mu_{Q,\b=2}^h(a)},\dots,a+\frac{t_k}{N\mu_{Q,\b=2}^h(a)}\rb\\
&=\det
\left[\frac{\sin\left(\pi(t_i-t_j)\right)}{\pi(t_i-t_j)}\right]_{1\leq i,j\leq k}
\end{align*}
uniformly in $t_1,\dots,t_k$ from any compact subset of $\R^k$ and uniformly in the point $a$ from any
compact proper subset of the support of $\mu_{Q,\b=2}^h$. This locally uniform convergence of the marginal densities was inherited
from
strong results on universality of unitary invariant (i.e. $\b=2$) ensembles (cf. \cite{LevinLubinsky08}). In order to apply these
results, we developed a method to express the correlation functions of the model $P_{N,Q,\b=2}^h$ as a
probabilistic mixture of unitary invariant ensembles with potential $V+f/N$, where $V$ was fixed and $f$ was random. However, this
representation was only possible for negative semi-definite $h$ and an argument involving complex analysis had to be used to extend
the universality for more general $h$.

So far, the local relaxation flow approach due to Erd\H{o}s,
Schlein and Yau (refined by others) \cite{ErdosSchleinYau} and applied to $\b$-Ensembles by Bourgade, Erd\H{o}s and Yau
\cite{Bourgadeetal1, Bourgadeetal} is the only method for
showing bulk universality for general $\b$-Ensembles. A remark on some crucial points of this method is
included in Section 4. Their approach actually addresses universality of gap distributions which implies the weaker form of
universality of the correlation measures as stated in Theorem \ref{local}. As we use their method, we obtain the same form of
convergence. If other sufficiently general universality results on $\b$-Ensembles yielding
stronger types of convergence were available, the method of \cite{GoetzeVenker} could be used to prove Theorem \ref{local} with
stronger forms of convergence. One advantage of the local relaxation flow approach is the possibility to compare local statistics of
eigenvalue ensembles and other, not necessary spectral ensembles, directly. This allows us to give a short proof of Theorem
\ref{local}.

Theorems \ref{global} and \ref{local} rely on comparison with a $\b$-Ensemble which has the same global asymptotics. This ensemble
is constructed in Section 2.
Section 3 contains the proof of Theorem \ref{local} via the local relaxation flow approach.
In Section 4, we give a short proof of Theorem \ref{global}.

\section{The associated invariant ensemble}
The main idea for the analysis of $P_{N,Q,\b}^h$ is to find a $\b$-Ensemble having the same global asymptotics.
In this short section we review the determination of the limiting measure for our particle system from \cite{GoetzeVenker} and use
this to construct an ensemble of eigenvalues with
the same global and local behaviour.

Let $\b>0$, $h$ be a continuous even function, $Q$ a strictly
convex
even
function and assume that
 \begin{align} 
 P_{N,Q,\b}^h(x):=\frac{1}{Z_{N,Q,\b}^h}\prod_{1\leq i<j\leq N}\lv x_i-x_j\rv^\b e^{-N\sum_{j=1}^NQ(x_j)-\sum_{i<j}h(x_i-x_j)},
\label{modelh}
 \end{align}
 defines the density of a probability measure on $\R^N$, where
 \begin{align*}
Z_{N,Q,\b}^h:=\int_{\R^N} \prod_{1\leq i<j\leq N}\lv x_i-x_j\rv^\b e^{-N\sum_{j=1}^NQ(x_j)-\sum_{i<j}h(x_i-x_j)}dx
 \end{align*}
denotes the normalizing constant.
We will use the notation
\begin{align}
 f_\mu(s):=\int f(t-s)d\mu(t),\ \  f_{\mu\mu}:=\int\int f(t-s)d\mu(t)d\nu(s)\label{beta7}
\end{align}
for a probability measure $\mu$ and an even function $\map{f}{\R}{\R}$ of sufficient integrability.

Using notation \eqref{beta7}, we make the Hoeffding type decomposition
\begin{align}
&\sum_{i<j}h(x_i-x_j)\nonumber\\
&=-\frac{N^2}{2}h_{\mu\mu}-\frac{N}{2}h(0) +N\sum_{j=1}^N
h_\mu(x_j)+\frac{1}{2}\lbb\sum_{i,j=1}^Nh(x_i-x_j)-\left[h_\mu(x_i)+h_\mu(x_j)-h_{\mu\mu}\right]\rbb\nonumber\\
&=-\frac{N^2}{2}h_{\mu\mu}-\frac{N}{2}h(0)+N\sum_{j=1}^N h_\mu(x_j)-\U(x),\quad \quad \text{ where}\nonumber\\
 &\U(x):=-\frac{1}{2}\lbb\sum_{i,j=1}^Nh(x_i-x_j)-\left[h_\mu(x_i)+h_\mu(x_j)-h_{\mu\mu}\right]\rbb.\label{51}
\end{align}
Now we can rewrite $\Ph$ as 
\begin{align}
\Ph(x)=\frac{1}{Z_{N,V_\mu,\b,\,\U}}\prod_{1\leq i<j\leq N}\lv x_i-x_j\rv^\b e^{-N\sum_{j=1}^NV_\mu(x_j)+\U(x)},\label{rewritten}
\end{align}
where we defined the external
field 
\begin{align*}
V_\mu(t):=Q(t)+h_\mu(t) 
\end{align*}
and absorbed the constant $\exp\{-(N^2/2)h_{\mu\mu}-(N/2)h(0)\}$ into the new normalizing constant
$Z_{N,V_\mu,\b,\,\U}$. 

Recall that the unique minimizer of the functional 
\begin{align*}
 I_{V,\b}(\mu):=\int V(t)d\mu(t)+\frac{\b}{2}\int\int \log\lv s-t\rv^{-1}d\mu(s)d\mu(t)
\end{align*}
is called equilibrium measure to the external field $V$ (and $\b>0$).  In \cite{GoetzeVenker} it has been shown that, provided $Q$
and $h$ are twice differentiable, $h$ is bounded and 
$h''\geq-\a_Q$, there is a measure $\mu$ such that $\mu$ is the equilibrium measure
to $V_\mu$. The uniqueness of such a $\mu$ follows (for $\a_Q$ large enough) from the convergence of $\rho_{N,Q,\b}^{h,1}$ towards
$\mu$. This measure is denoted $\mu_{Q,\b}^h$ in Theorem \ref{global}, but for brevity we will simply write $\mu$ instead of
$\mu_{Q,\b}^h$ and skip the indices $\mu$ in \eqref{rewritten}.

We note in passing that the external field $V=Q+h_\mu$ is convex (due to $h''\geq-\a_Q$), even and real-analytic.

 We will often use representation  \eqref{rewritten}. The proofs of Theorems
\ref{global}
and \ref{local} rely on
comparison of $\Ph$ with the $\b$-Ensemble
\begin{align}
P_{N,V,\b}(x)=\frac{1}{Z_{N,V,\b}}\prod_{1\leq i<j\leq N}\lv x_i-x_j\rv^\b e^{-N\sum_{j=1}^NV(x_j)}. \label{eq2}
\end{align}

\section{Proof of Theorem \ref{local}}

In this section we use the local relaxation flow approach developed by Erd\H{o}s, Yau, Schlein et. al. to establish universality of
the
local bulk correlations. First we introduce some notation from \cite{Bourgadeetal1}.

Let $k$ be fixed. Let $\map{G}{\R^k}{\R}$ be a smooth function with compact support and $m=(m_1,\dots,m_k)$ with $m_j$ being
positive integers. Define
\begin{align*}
 G_{i,m}(x):=G(N(x_i-x_{i+m_1}),\dots,N(x_{i+m_{k-1}}-x_{i+m_k})).
\end{align*}
The Dirichlet form of a smooth test function $\map{f}{\R^N}{\R}$ w.r.t. a probability measure $d\o$ on $\R^N$ is defined as
\begin{align*}
 &D_{\w}(f):=\frac{1}{2N}\sum_{j=1}^N \int \lbb\partial_{x_j}f\rbb^2d\o.
\end{align*}
Let $f$ be a probability density function w.r.t. $d\o$. The (relative) entropy of $f$ w.r.t. $d\o$ is defined as
\begin{align*}
&S_{\w}(f):=\int f\log fd\o.
\end{align*}

We will use the following general theorem.

\begin{prop}\cite[Lemma 5.9]{Bourgadeetal1}\label{erdosyau}
Let $\map{G}{\R^k}{\R}$ be bounded and of compact support. Let $d\o$ be a probability measure on
$\{x\,:\,x_1<x_2<\dots<x_N\}\subset\R^N$ given
by
\begin{align*}
	d\o=\frac{1}{Z} e^{-\b N\hat{\mathcal{H}}(x)}dx,\quad \hat{\mathcal{H}}(x)=\mathcal{H}_0(x)-\frac{1}{N}\sum_{i<j}\log\lv x_j-x_i\rv
\end{align*}
with the property that $\nabla^2\mathcal{H}_0\geq \tau^{-1}$ holds for some positive constant $\tau$. Let $qd\o$ be another
probability measure with smooth density $q$. Let $J\subset \{1,2,\dots,N-m_k-1\}$ be a set of indices. Then for any $\e_1>0$ we have
\begin{align*}
\lv \frac{1}{\lv J\rv}\sum_{i\in J}\int G_{i,m}\,qd\o-\frac{1}{\lv J\rv}\sum_{i\in J}\int G_{i,m}d\o\rv\leq
C\sqrt{N^{\e_1}\frac{D_{\o}(\sqrt{q})\tau}{\lv J\rv}}+C\sqrt{S_{\o}(q)}e^{-cN^{\e_1}}.
\end{align*}
\end{prop}

In our application we will choose $d\o=P_{N,V,\b}$ and $q=(Z_{N,V,\b}/Z_{N,V,\b,\U})\exp\{\U\}$. Formally, we should replace $V$ by
$V/\b$. However, for notational convenience we will omit this trivial scaling. If $\a_Q$ is large enough, then
$V$ is strongly convex, hence $\tau=1/\a_V$. By the symmetry of $\Ph$ and $P_{N,V,\b}$, it is equivalent to restrict the measure to
the simplex
$\{x\,:\,x_1<x_2<\dots<x_N\}$ and multiply by $N!$\,.
From \cite[Theorem 2.3]{ErdosYau12} we have that 

\begin{align*}
 S_{\o}(q)\leq {C}D_{\o}(\sqrt{q}).
\end{align*}
It is thus sufficient to prove that $D_{\o}(\sqrt{q})$ is bounded in $N$ as $J$ will be chosen such that $\lv J\rv\sim N$, in order
to identify the bulk correlations.

\begin{remark}[On the local relaxation flow approach]
 To briefly explain the essence of this method due to Erd\H{o}s, Schlein, Yau and others (see e.g. \cite{ErdosYau12} for
references and a complete review), let us consider two measures as in Proposition \ref{erdosyau}, $d\w$ and $qd\w$ and their
statistics $\int gd\w$ and $\int g\,qd\w$ for some test function $g$. Assume that one can define a Markov process on $\R^N$ in
terms of the Dirichlet form $D_\w$ (or the formal generator $L_N:=\frac{1}{2N}\Delta-\frac{1}{2}(\nabla
\hat{\mathcal{H}})\nabla$), having $d\o$ as stationary distribution. Assume that the process has the initial
distribution $qd\o$ and denote the evolution of the density w.r.t. $d\w$ by $(f_t)_{t\geq 0}, f_0=q, f_\infty=1$.
Then one can write 
\begin{align*}
 \int g\,qd\w-\int gd\w=\lbb\int g\,qd\w-\int g\,f_td\w\rbb+\lbb \int g\,f_td\w-\int gd\w \rbb,	
\end{align*}
which corresponds to running the process up to time $t$. If the process is ergodic and the time $t$ is large enough, $\int
g\,f_td\w$ will be close to the
equilibrium $\int g\,d\w$. If this $t$ is still ``small'', i.e. the convergence to the stationary distribution is fast, then the
distance between $\int g\,qd\w$ and $\int g\,f_td\w$ should be not too big. These distances are measured in terms of Dirichlet form
and
entropy of $d\w$. These estimates are due to the Bakry-Emery method
which yields ergodicity or \textit{relaxation} making use of the strict convexity of the Hamiltonian, i.e. of the bound
$\nabla^2\mathcal{H}_0\geq \tau^{-1}$. It turns out that the constant $\tau$ is the time scale for the relaxation to equilibrium,
meaning that e.g. $S_\w(f_t)\leq e^{-t/\tau}S_\w(f_0)$. Here we tacitly used that the logarithmic part of the Hamiltonian
$\hat{\mathcal{H}}$ is convex, therefore does not increase the relaxation time. However, one crucial observation is
that from the trivial bound
\begin{align*}
 \langle v,\nabla^2 \hat{\mathcal{H}}(x)v\rangle\geq \frac{1}{\tau}\|v\|^2+\frac{1}{N}\sum_{i<j}\frac{(v_i-v_j)^2}{(x_i-x_j)^2}
\end{align*}
one can infer that the relaxation is much faster in the directions $(v_i-v_j)$ provided that $x_i$ and $x_j$ are close. Indeed, the
mean distance between neighboring eigenvalues is of order $1/N$, hence the convexity bound for the Hamiltonian should be locally of
order $N$, therefore yielding a time to the local equilibrium of order $1/N$ whereas the time to the global equilibrium is of order
1. This informal reasoning can be captured by choosing test functions like $G_{i,m}$ which depend only on eigenvalue
differences in the local scaling (i.e. multiplied by $N$) and vanish whenever two eigenvalues are not close to each other. By
exploiting these features of $G_{i,m}$ and some estimates, one arrives at Proposition \ref{erdosyau}. For arbitrary test functions
$g$, one would get basically the same estimate except for the quantity $\lv J\rv\sim N$ which divides $D_\w(\sqrt{q})$. 

One
problem with this idea is that the existence of the process associated to the Dirichlet form is not clear for $\b\in(0,1)$. For
$\b\geq
1$, the repulsion is strong enough to prevent collision between the eigenvalues but for $\b<1$ the probability of explosion is
positive. This problem was overcome in \cite{erdosgraph} by smoothing the singular logarithmic term and using the approach above
for the corresponding process.

In our application, we can effectively estimate the Dirichlet form. This is due to the fact that $d\w$ and $qd\o$ have the same
global limit (cf. Theorem \ref{global}) and we have concentration of $\U$ under
$d\o=P_{N,V,\b}$ (cf. Proposition \ref{Fourier} below). 
\end{remark}
Eventually we will prove the following key proposition.
\begin{prop}\label{Dirichlet}
 Let $D_{N,V,\b}$ denote the Dirichlet form w.r.t. $P_{N,V,\b}$ and \\$q=(Z_{N,V,\b}/Z_{N,V,\b,\U})\exp\{\U\}$. Then there is a
constant $C$ such that we have for
$\a_Q$ large
enough
\begin{align*}
 D_{N,V,\b}(\sqrt{q})\leq C \quad \text{for all }N.
\end{align*}
\end{prop}

One ingredient to the proof of Proposition \ref{Dirichlet} is the following identity from \cite{GoetzeVenker}, which can be
obtained using Fourier inversion.
We have
 \begin{align}\label{lemma4}
  \U(x)=-\frac{1}{2\sqrt{2\pi}}\int\lv \cir{u}_N(t,x)\rv^2\
\F{h}(t)dt ,\qquad \text{where}
 \end{align}
\begin{align*}
 &\cir{u}_N(t,x):=\sum_{j=1}^N \cos(tx_j)-N\int \cos(ts)d\mu(s)+\sqrt{-1}\sum_{j=1}^N
\sin(tx_j),\ \F{h}(t):=\frac{1}{\sqrt{2\pi}}\int_{\R}e^{-its}h(s)ds.
\end{align*}
 A trivial but useful observation from \eqref{rewritten} and \eqref{eq2} is
\begin{align*}
 &\E_{N,Q,\b}^hf(x)=(Z_{N,V,\b}/{Z_{N,V,\b, \U}})\E_{N,V,\b}f(x)e^{\U(x)}.
\end{align*}

The next proposition establishes concentration of $\U$.
\begin{prop}\label{Fourier} For each $\l>0$ there are $\a^h(\l)>0, C_1(\l),
C_2(\l)$ such that for all $N$
\begin{align*}
  0<C_1(\l)<\E_{N,V,\b}\exp\{\l\U(x)\}\leq C_2(\l)<\infty,\quad \text{ if } \a_Q\geq \a^h(\l).
\end{align*}
In particular, we have for $\a_Q$ large enough for all $N$
\begin{align*}
 0<C_1(1)\leq Z_{N,V,\b, \U}/{Z_{N,V,\b}}\leq C_2(1)<\infty.
\end{align*}

\end{prop}
The proof is based on identity \eqref{lemma4} and the following  concentration of measure inequality for linear statistics. Details
can be found in \cite{GoetzeVenker}. By $\mu_{Q,\b}$ we will denote the
equilibrium measure to the external field $Q$ and $\b$.

\begin{lemma}\label{Concentration2}
Let $Q$ be a real analytic external field with $Q''\geq c>0$. Then
for any Lipschitz function $f$ whose third derivative is bounded on a neighborhood of $\textup{supp}(\mu_{Q,\b})$, we have for any
$\epsilon>0$
\begin{align*}
\E_{N,Q,\b}\exp\lee{\epsilon\lbb\sum_{j=1}^N f(x_j)-N\int f(t)d\mu_{Q,\b}(t)\rbb}\ree\leq \exp\lee{\frac{\epsilon^2\Lip{f}^2}{2c}}+
\e C (\|f\|_\infty+\|f^{(3)}\|_\infty)\ree,
\end{align*}
where $\Lip{f}$ denotes the Lipschitz constant of $f$ on $\R$ and $\|f\|_\infty$ $(\|f^{(3)}\|_\infty)$ denotes the bound of (the
third derivative of) $f$ on the neighborhood of $\textup{supp}(\mu_{Q,\b})$. Here the constant $C$ does not depend on $N$ or $f$.
\end{lemma}

The lemma follows from an application of a basic logarithmic Sobolev inequality due to the strict convexity of the external
field $Q$, to the Lipschitz function $\sum_{j=1}^N f(x_j)$ (see e.g. \cite{AGZ}) and a rate of convergence result from
\cite{Shcherbina} which allows to replace the exact mean by its limit as $N\to\infty$.

\begin{proof}[Proof of Proposition \ref{Dirichlet}]
 The ratio $Z_{N,V,\b}/Z_{N,V,\b,\U}$ is bounded by Proposition \ref{Fourier} and therefore negligible. 
We have by H\"older's inequality for $\e>0$
\begin{align*}
 &D_{N,V,\b}(\sqrt{q})\leq
C\frac{1}{2N}\sum_{l=1}^N\E_{N,V,\b}\lbb\partial_{x_l}\exp\{\frac{1}{2}\U(x)\}\rbb^2=C\frac{1}{8N}\sum_{l=1}^N\E_{N,V,\b}\exp\{
\U(x)\} \lbb\partial_ {x_l} \U(x)\rbb^2\\
&\leq C\lbb \E_{N,V,\b}\exp\{(1+\e)
\U(x)\}\rbb^{1/(1+\e)} \frac{1}{8N}\sum_{l=1}^N\lbb\E_{N,V,\b}\lvb\partial_ {x_l} \U(x)\rvb^{2(\e+1)/\e} \rbb^{\e/(\e+1)}.
\end{align*}
Again by Proposition \ref{Fourier}, $\lbb \E_{N,V,\b}\exp\{(1+\e)
\U(x)\}\rbb^{1/(1+\e)}$ is bounded in $N$. In order to bound the second term, recall that
\begin{align*}
\U(x)=- \frac{1}{2\sqrt{2\pi}}\int \lbb\big\lvert\sum_{j=1}^N \cos(tx_j)-N\int \cos(ts)d\mu(s)
\big\rvert^2+\lvert\sum_{j=1}^N
\sin(tx_j) \rvert^2\rbb\F{h}(t)dt.
\end{align*}
In the following we only treat the cosine term, the term involving the sine can be estimated analogously. We have
\begin{align}
 &\lvb\partial_{x_l}\int \big\lvert\sum_{j=1}^N \cos(tx_j)-N\int
\cos(ts)d\mu(s)\big\rvert^2\F{h}(t)dt\rvb^{2(\e+1)/\e}\nonumber\\
&=\lvb 2\int \lbb \sum_{j=1}^N
\cos(tx_j)-N\int \cos(ts)d\mu(s)\rbb t\sin(tx_l)\F{h}(t)dt\rvb^{2(\e+1)/\e}\nonumber\\
&\leq C \int \lvb \sum_{j=1}^N
\cos(tx_j)-N\int \cos(ts)d\mu(s)\rvb^{2(\e+1)/\e} \lv t\rv^{2(\e+1)/\e}\lvb\F{h}(t)\rvb dt\label{beta1}
\end{align}
where the last inequality is derived by first applying the triangle inequality and then using Jensen's inequality.
Lemma \ref{Concentration2} gives that the absolute moments of $\sum f(x_j)-N\int fd\mu$ are bounded by those of a certain Gaussian
distribution with mean of order $\|f\|_\infty+\|f^{(3)}\|_\infty$ and variance of order $\Lip{f}^2$ (times a factor of order
$\Lip{f}$). We thus get
\begin{align*}
 \E_{N,V,\b}\lvb \sum_{j=1}^N
\cos(tx_j)-N\int \cos(ts)d\mu(s)\rvb^{2(\e+1)/\e}\leq p(t)
\end{align*}
for some polynomial $p$. By the strong
decay of $\F{h}$, the expectation of
\eqref{beta1} is bounded in $N$. This gives the claimed bound.
\end{proof}

\begin{proof}[Proof of Theorem \ref{local}]
 From Propositions \ref{erdosyau} and \ref{Dirichlet} we have that the statistics $\frac{1}{\lv J\rv}\sum_{i\in
J}\E_{N,Q,\b}^hG_{i,m}$ and $\frac{1}{\lv J\rv}\sum_{i\in
J}\E_{N,V,\b}G_{i,m}$ coincide in the limit $N\to\infty$, as long as $\lim_{N\to\infty}\frac{N^{\e_1}}{\lv
J\rv}=0$ for some $\e_1>0$. It
is a standard argument (\cite[Section 7]{ErdosSchleinYauYin}) to infer from this that also the correlation measures of $\Ph$
and $P_{N,V,\b}$ coincide in
the sense
of Theorem \ref{local}. To give the idea of this argument, note that by a simple rescaling we have the identity
\begin{align}
 &\int f(t_1,\dots,t_k)\int_{a-s_N}^{a+s_N}\rho_{N,Q,\b}^{h,k}\lbb
u+\frac{t_1}{N\mu(a)},\dots,u+\frac{t_k}{N\mu(a)}\rbb\,\frac{du}{2s_N}dt_1\dots dt_k\nonumber\\
&=(1+o(1))\int_{a-s_N}^{a+s_N}\int\sum_{i_1\not=i_2\not=\dots\not=i_k}
\tilde{f}(N(x_{i_1}-u),N(x_{i_1}-x_{i_2}),\dots,N(x_{i_{k-1}}-x_{i_k}))qd\w\frac{du}{2s_N},\label{gaptofunction}
\end{align}
where we use the notation from Proposition \ref{erdosyau}
and $$\tilde{f}(t_1,\dots,t_k):=f(\mu(a)t_1,\mu(a)(t_2-t_1),\dots,\mu(a)(t_k-t_{k-1})).$$ Symmetrizing and rearranging the
summation,
\eqref{gaptofunction} can be written as
\begin{align}\label{gap2}
(1+o(1))\int_{a-s_N}^{a+s_N}\int\sum_{m\in S_k}\sum_{i=1}^N Y_{i,m}(u,x)qd\w\frac{du}{2s_N},
\end{align}
where $S_k$ denotes the set of $(k-1)$-tuples of increasing positive integers, $m=(m_2,m_3,\dots,m_k)$ and 
\begin{align*}
 Y_{i,m}(u,x):=\tilde{f}(N(x_i-u),N(x_i-x_{m_2}),\dots,N(x_i-x_{m_k})).
\end{align*}
If $i+m_k>N$, then we set $Y_{i,m}:=0$.
Now, one can show that as $N\to\infty$, \eqref{gap2} can be replaced by $\int \sum_{m\in S_k}\frac{1}{N}\sum_{i=1}^N G_{i,m}qd\w$,
where $G(t_2,\dots,t_k):=\int_\R \tilde{f}(u,t_2,\dots,t_k)du$. Then Proposition \ref{erdosyau} can be applied for each fixed $m$.
For details see \cite[Section 7]{ErdosSchleinYauYin}.

It remains to see that the limits of the correlation measures of $P_{N,V,\b}$ are indeed universal and in particular coincide with
the Gaussian ones. This is the universality result
\cite[Corollary 2.2]{Bourgadeetal1} which precisely states that the correlation measures of $P_{N,Q_1,\b}$ and $P_{N,Q_2,\b}$
have the same limit (in the sense of Theorem \ref{local}) for any real analytic and strongly convex $Q_1,Q_2$ with
$\a_{Q_1},\a_{Q_2}>0$.
\end{proof}

\section{Proof of Theorem \ref{global}}
\begin{proof}[Proof of Theorem \ref{global}]
Recall that $\mu$ was determined such that $\mu$ is the equilibrium measure to $V=Q+h_\mu$ and $\b$. It remains to show that
$\mu$ is uniquely determined by this requirement and indeed the limit of the first correlation function.
 We consider a Lipschitz function $\map{f}{\R}{\R}$ with three continuous derivatives and estimate for any $\e>0$
\begin{align*}
 P_{N,Q,\b}^h(\lvert N^{-1}\sum_{j=1}^Nf(x_j)-\int fd\mu\rvert>\e)=(Z_{N,V,\b}/{Z_{N,V,\b,\,
\U}})\E_{N,V,\b}e^{\U(x)}\dopp{1}_{\{\lvert
N^{-1}\sum_{j=1}^Nf(x_j)-\int fd\mu\rvert>\e\}}.
\end{align*}
By H\"older's inequality and Proposition \ref{Fourier}, we have
\begin{align*}
 P_{N,Q,\b}^h(\lvert N^{-1}\sum_{j=1}^Nf(x_j)-\int fd\mu\rvert>\e)\leq C \lbb P_{N,V,\b}(\lvert N^{-1}\sum_{j=1}^Nf(x_j)-\int
fd\mu\rvert>\e)\rbb^c
\end{align*}
for some $c,C>0$. By Lemma \ref{Concentration2}, this last probability converges for any $\e>0$ to $0$ exponentially fast as
$N\to\infty$. We
conclude that $$\lim_{N\to\infty}\E_{N,Q,\b}^h\lvert N^{-1}\sum_{j=1}^Nf(x_j)-\int fd\mu\rvert=0\ \text{ 
and hence }\
\lim_{N\to\infty}\E_{N,Q,\b}^h N^{-1}\sum_{j=1}^Nf(x_j)=\int fd\mu.$$ As convergence for smooth
Lipschitz functions determines weak convergence, the weak convergence of the first correlation measure follows.
 As the limit of weak convergence is unique, this shows uniqueness of $\mu$. It is known that the real-analyticity and convexity
of $V$ ensures the existence and positivity of
the continuous density of $\mu$ (see e.g. \cite{McLaughlinMiller08}).
\end{proof}

\section*{acknowledgement}
 The author would like to thank the reviewers for helpful comments which improved the presentation.

\bibliographystyle{alpha}
\bibliography{bibliography}

\end{document}